\documentclass{article}
\usepackage[utf8]{inputenc}
\usepackage{amsmath}
\usepackage{amssymb}
\usepackage{amsthm}
\usepackage{graphicx}
\usepackage[margin=1.4cm]{geometry}

\newtheorem{thm}{Theorem}

\newtheorem{prop}{Proposition}
\newtheorem{lem}{Lemma}
\newtheorem{rmk}{Remark}

\title{$SU(2)$-invariant steady gradient Ricci solitons on four-manifolds}
\author{Timothy Buttsworth}
\date{}

\begin{document}
\maketitle
\abstract{Using center manifolds and topological degree theory, we construct a new family of complete, $SU(2)$-invariant and steady gradient Ricci solitons on the four-dimensional non-compact cohomogeneity one manifold with group diagram $\mathbb{Z}_4\subset U(1)\subset SU(2)$. We also provide simpler constructions of the existing $U(2)$-invariant steady and complete gradient solitons on the cohomogeneity one manifolds with group diagrams $\mathbb{Z}_n\subset U(1)\subset SU(2)$ for any $n\in \mathbb{N}$, including Appleton's non-collapsed solitons for $n\ge 3$.}
\section{Introduction}
A Riemannian manifold $(M,g)$ is called a \textit{Ricci soliton} if there is a vector field $X$ so that 
\begin{align*}
Ric(g)+\mathcal{L}_X g=\lambda g
\end{align*}
for some $\lambda\in \mathbb{R}$. 
A Ricci soliton generates a solution of the Ricci flow which evolves only via scalings and pullbacks of the time-varying diffeomorphism generated by the vector field $X$.
If $\lambda$ is positive, zero, or negative, we say the soliton is \textit{shrinking}, \textit{steady} or \textit{expanding}, respectively. If the vector field $X$ appears as the gradient of some scalar potential function $u:M\to \mathbb{R}$, we say that the soliton is \textit{gradient}. 

A classical example of a gradient steady Ricci soliton that is \textit{not} Ricci-flat is the \textit{Bryant soliton} on $\mathbb{R}^3$. This soliton and its higher-dimensional analogues are all \textit{rotationally-invariant}, i.e., the geometry is foliated by round spheres of dimension one less than that of the ambient manifold. These Bryant solitons are thus examples of \textit{cohomogeneity one} Riemannian manifolds, i.e., their geometries are foliated by homogeneous Riemann spaces of dimension one less than that of the ambient manifold. The search for steady solitons amongst the class of cohomogeneity one Riemannian manifolds has been quite fruitful (see, for example, the solitons constructed in \cite{Dancer13}). 

For four-dimensional manifolds, a particularly popular avenue of study is the class of Ricci solitons that are foliated by a one-parameter family of homogeneous Berger sphere geometries on $\mathbb{S}^3/\mathbb{Z}_n$, and can be smoothly completed by including an $\mathbb{S}^2$ `bolt' in the topology; since there is only one singular orbit, the resulting four-dimensional manifold $M$ is non-compact. We will carefully describe these manifolds in the next section. 
When $n=1,2$, we have the well-known Ricci-flat Taub-Bolt and Eguchi-Hanson metrics (described in, for example, \cite{TaubBolt} and \cite{Eguchi}, respectively).  For any $n\in \mathbb{N}$, there is also a one-parameter family of steady gradient solitons that are \textit{not} Ricci flat.  In fact, for $n\ge 3$, Appleton showed \cite{Appleton} that taking a limit of these solitons gives rise to a complete steady Ricci soliton that is not Ricci-flat \textit{and non-collapsed}; this construction breaks down for $n=1,2$ precisely because of the complete Ricci-flat metrics that are available.

Since the foliation of these solitons is by Berger spheres, all of these examples are invariant under a certain action of $U(2)$  which acts transitively on the $\mathbb{S}^3/\mathbb{Z}_n$ fibers. However, it is also true that the Lie subgroup $SU(2)\subset U(2)$ acts transitively on these same fibers; the aim of this paper is to examine the possible Ricci solitons that can arise under the weaker assumption of $SU(2)$ invariance instead of $U(2)$ invariance. In this paper, we provide simpler constructions of the well-known $U(2)$-invariant solitons, and use these to subsequently construct solitons that are \textit{not} $U(2)$-invariant. 
\begin{thm}\label{mainexistencetheorem}
For $n=4$, there exists a one-parameter family of $SU(2)$-invariant steady gradient Ricci solitons on $M$ that are \textit{not} $U(2)$-invariant.
\end{thm}
The new solitons are found by perturbing Appleton's non-collapsed soliton in the `non $U(2)$-invariant' direction. Our perturbative tool is topological degree theory, rather than the inverse function theorem, and the complications arising from the non-compactness of the manifold are primarily dealt with through the theory of center manifolds for dynamical systems.
\section*{Acknowledgements}
I am grateful to Ramiro Lafuente, Artem Pulemotov and Wolfgang Ziller for comments on Sections 1 and 2.
\section{Preliminaries}
In this section, we more carefully examine the four-dimensional manifolds on which we study the Ricci soliton equation. We also look at the possible structure that Ricci solitons can have on the principal part of the manifold, and state the conditions we need to impose in order to ensure that the metric  can be smoothly extended to the $\mathbb{S}^2$ singular orbit. 
\subsection{Algebra and topology}
Consider the Lie group 
\begin{align*}
SU(2)=\Bigg\{\begin{pmatrix}
\alpha&-\overline{\beta}\\
\beta&\overline{\alpha}
\end{pmatrix} \vert \alpha,\beta\in \mathbb{C}, \left|\alpha\right|^2+\left|\beta\right|^2=1 \Bigg\},
\end{align*}
and the closed Lie subgroups
\begin{align*}
\mathbb{Z}_n=\Bigg\{\begin{pmatrix}
e^{\frac{2\pi i k}{n}}&0\\
0&e^{-\frac{2\pi i k}{n}}
\end{pmatrix} \ \vert \  k=0,\cdots,n-1\Bigg\}\subset U(1)=\Bigg\{\begin{pmatrix}
e^{i\theta}&0\\
0&e^{-i \theta}
\end{pmatrix} \ \vert \  \theta\in [0,2\pi)\Bigg\}\subset SU(2).
\end{align*}
We have that 
$U(1)/\mathbb{Z}_n=\mathbb{S}^1$; this $U(1)$-action extends to a cohomogeneity one action on $\mathbb{R}^2$ which preserves and acts transitively on circles centered at the origin. Also note that $SU(2)/U(1)=\mathbb{S}^2$. 
The four-dimensional manifolds we will consider are given by $$M=SU(2)\times_{U(1)} \mathbb{R}^2,$$ where $(g,x)\sim (g k^{-1},kx)$ for all $(g,x,k)\in SU(2)\times \mathbb{R}^2\times U(1)$. The action of $SU(2)$ on $M$ is cohomogeneity one, with principal isotropy of $\mathbb{Z}_n$, and a unique singular orbit with isotropy $U(1)$ occuring where $x=0$. Topologically, these manifolds appear as $(0,\infty)\times SU(2)/ \mathbb{Z}_n$, completed with a copy of $SU(2)/U(1)=\mathbb{S}^2$ (a `bolt') at the origin. 

We find it handy to note that the Lie algebra $\mathfrak{su}(2)$ of $SU(2)$ is given by the real span of the three vectors 
\begin{align*}
u_1=\begin{pmatrix}
i&0\\
0&-i
\end{pmatrix}, \qquad u_2=\begin{pmatrix}
0&i\\
i&0
\end{pmatrix}, \qquad u_3=\begin{pmatrix}0&-1\\
1&0\end{pmatrix}
\end{align*}
which satisfy the Lie bracket relations 
\begin{align*}
[u_1,u_2]=2u_3, \qquad [u_2,u_3]=2u_1, \qquad [u_3,u_1]=2u_2. 
\end{align*}
Thus, for an element $u_{\theta}=\begin{pmatrix}e^{i\theta}&0\\
0&e^{-i \theta}\end{pmatrix}\in U(1)$, we have 
\begin{align}\label{u1adjointaction}
Ad_{u_{\theta}}(u_1)=u_1, \qquad Ad_{u_{\theta}}(\cos(\psi)u_2+\sin(\psi)u_3)=\cos(\psi+2\theta)u_2+\sin(\psi+2\theta)u_3,
\end{align}
and $\exp(\theta u_1)=u_{\theta}$. 

\subsection{Geometry on $(0,\infty)\times SU(2)/\mathbb{Z}_n$}
Take an open interval $I$ and an $SU(2)$-invariant Riemannian metric on $I\times SU(2)/\mathbb{Z}_n$. By finding a geodesic which transverses all of the $SU(2)$ orbits perpendicularly, we can assume that, up to scaling and diffeomorphism, our Riemannian metric has the form 
\begin{align}\label{geodesicmetricform}
dr\otimes dr+g(r),
\end{align}
where $\{g(r)\}_{r\in I}$ is a one-parameter family of $SU(2)$-invariant Riemannian metrics on $SU(2)/\mathbb{Z}_n$. 
We now need to understand what these homogeneous Riemannian metrics look like. 
\begin{prop}\label{homogeneousdiagonal}
Let $g$ be a left-invariant Riemannian metric on $SU(2)$. Then there are left-invariant vector fields $X_1,X_2,X_3$ satisfying 
\begin{align}\label{Milnorframe}
[X_1,X_2]=X_3, \qquad [X_2,X_3]=X_1, \qquad [X_3,X_1]=X_2,
\end{align} (called a Milnor frame) and positive numbers $f_1,f_2,f_3$ so that 
\begin{align}\label{leftinvariantsu2}
g=f_1^2 \omega_1\otimes \omega_1+f_2^2 \omega_2\otimes \omega_2+f_3^2 \omega_3\otimes \omega_3,
\end{align}
where $\omega_i$ is the one-form corresponding to $X_i$. 
This Riemannian metric passes to an $SU(2)$-invariant Riemannian metric on $SU(2)/\mathbb{Z}_n$ if and only if one of the following conditions holds:
\begin{itemize}
\item $n=1,2$; 
\item $n=4$, and after possibly altering our choice of diagonalising Milnor frame, we have that $X_1$ is a scalar multiple of $u_1$, and $\text{span}\{X_2,X_3\}=\text{span}\{u_2,u_3\}$;
\item $n\in\mathbb{N}\setminus\{1,2,4\}$, we can choose our diagonalising Milnor frame with $X_i=\frac{u_i}{2}$ and in this basis, we have $f_2=f_3$.
\end{itemize}
\end{prop}
\begin{proof}
The existence of a diagonalising Milnor frame for each left-invariant metric on $SU(2)$ is well-known, and essentially follows from the fact that the automorphism group of $\mathfrak{su}(2)$ is $SO(3)$. For the claims about passing these left-invariant Riemannian metrics to the quotient, we will work at the level of Lie algebras. Indeed, \eqref{leftinvariantsu2} defines an inner product on $\mathfrak{su}(2)$; this will define an invariant metric on $SU(2)/\mathbb{Z}_n$ if and only if this inner product is invariant under the action of $Ad_{\mathbb{Z}_n}$. According to 
\eqref{u1adjointaction}, if $n=1,2$, then $Ad_{u_{\theta}}$ acts as the identity for each $u_{\theta}\in\mathbb{Z}_n$, so all metrics of the form \eqref{leftinvariantsu2} are invariant in the quotient. If $n=4$, then two elements $u_{\theta}\in \mathbb{Z}_n$ act as the identity, and the other two elements operate as 
\begin{align*}
Ad_{u_{\theta}}(u_1)=u_1, \qquad Ad_{u_{\theta}}(u_2)=-u_2, \qquad Ad_{u_{\theta}}(u_3)=-u_3. 
\end{align*} 
This implies that the metric is invariant if and only if $u_1$ is orthogonal to $u_2$ and $u_3$. By scaling the $u_1,u_2,u_3$ basis by $\frac{1}{2}$ and applying an $SO(2)$ change to the $u_2,u_3$ basis, we obtain a Milnor frame diagonalising $g$. 
For other values of $n$, the abundance of adjoint actions fixing $u_1$ and rotating $u_2,u_3$ implies that we have invariance if and only if $u_1$ is orthogonal to $u_2,u_3$, and the metric restricted to the $u_2,u_3$ basis is a scalar multiple of the identity. We thus obtain our required Milnor frame simply with $X_i=\frac{u_i}{2}$. 
\end{proof}
We now make the interesting observation that if an $SU(2)$-invariant Riemannian metric on $I\times SU(2)/\mathbb{Z}_n$ is a gradient Ricci soliton with $SU(2)$-invariant potential function, then the choice of diagonalising Milnor frame can be made global. 
\begin{prop}\label{cohom1quotientmetric}
For an $SU(2)$-invariant Riemannian metric of the form \eqref{geodesicmetricform} on $I\times SU(2)$ which is a gradient Ricci soliton with $SU(2)$-invariant potential function $u$, there is a basis of vector fields $X_1,X_2,X_3$ on $SU(2)$ satisfying \eqref{Milnorframe} so that 
the metric has the form
\begin{align}\label{staticmetric}
dr\otimes dr+f_1(r)^2\omega_1\otimes \omega_1 +f_2(r)^2\omega_2\otimes \omega_2+f_3(r)^2\omega_3\otimes \omega_3,
\end{align}
where $f_i:I\to (0,\infty)$ are smooth functions. This metric passes to an $SU(2)$-invariant soliton on $I\times SU(2)/\mathbb{Z}_n$ if and only if one of the following is satisfied:
\begin{itemize}
\item $n=1,2$; 
\item $n=4$, and after possibly altering our choice of diagonalising Milnor frame, we have that $X_1$ is a scalar multiple of $u_1$, and $\text{span}\{X_2,X_3\}=\text{span}\{u_2,u_3\}$;
\item $n\in\mathbb{N}\setminus\{1,2,4\}$, we can choose our diagonalising Milnor frame with $X_i=\frac{u_i}{2}$ and in this basis, we have $f_2=f_3$.
\end{itemize}
\end{prop}
\begin{proof}
Since the function $u:I\times SU(2)$ only depends on $r\in I$, we have that $Hess(u)(\partial_r,X)=0$ for any $X\in \mathfrak{su}(2)$. Thus, for the metric in \eqref{geodesicmetricform} to be a gradient Ricci soliton, we have $Ric(\partial_r,X)=0$ as well.  Now choose some $r_0\in I$. For any Milnor frame $\{X_1,X_2,X_3\}$ which diagonalises $g(r_0)$, the condition $Ric(\partial_r,X)=0$ at the $r_0$ orbit becomes 
\begin{align}\label{crossRiccisu2}
(g^{ii}(r_0)-g^{jj}(r_0))g_{ij}'(r_0)=0
\end{align}
for all $i\neq j$ (see Section 3 of \cite{Dammerman}). 
Equation \eqref{crossRiccisu2} implies that there is a Milnor frame $\{X_1,X_2,X_3\}$ diagonalising both $g(r_0)$ and $g'(r_0)$. Indeed, if we choose an initial Milnor frame diagonalising $g(r_0)$ and $g^{ii}(r_0)\neq g^{jj}(r_0)$ for all $i\neq j$, then $g'$ is diagonal at $r_0$, by \eqref{crossRiccisu2}. If $g^{ii}=g^{jj}\neq g^{kk}$, then \eqref{crossRiccisu2} implies that $g'_{ik}(r_0)=g'_{jk}(r_0)=0$, and we can rotate the basis in $X_i,X_j$ to assume that $g'_{ij}(r_0)=0$, because such a change preserves the equalities of \eqref{Milnorframe}. If $g^{ii}=g^{jj}=g^{kk}$, then we can apply an $SO(3)$ change to ensure that $g'(r_0)$ is diagonal because such changes preserve \eqref{Milnorframe}. 

 Now, according to \cite{Dancer13} (for example), the other gradient Ricci soliton equations become
\begin{align}\label{solitonODESgeneric}
\begin{split}
-tr(L')-tr(L^2)+u''-\lambda&=0,\\
R(g)-tr(L)L-L'+u'L-\lambda I&=0,\\
g'-2gL&=0,
\end{split}
\end{align}
where $R(g)$ is the Ricci curvature operator of $g(r)$ on $SU(2)$. Since $L$ and $g$ are diagonal at the point $r_0$, and $R(g)$ is diagonal in a Milnor frame whenever the same is true of $g$, standard existence and uniqueness results for ODEs imply that $L$ and $g$ are diagonal in this same basis for all $r\in I$. 

We now examine when metrics of the form \eqref{staticmetric} pass to $SU(2)$-invariant Riemannian metrics on $I\times SU(2)/\mathbb{Z}_n$. If $n=1,2$, there are no obstructions.

 If $n=4$ and there is an alternate globally-diagonalising Milnor frame with $X_1$ equal to a multiple of $u_1$ and $X_2,X_3$ spanned by $u_2,u_3$, then Proposition \ref{homogeneousdiagonal} implies that this metric is well-defined and $SU(2)$-invariant in the quotient. On the other hand, if the metric passes to the quotient, then for each $r\in I$, there is choice of Milnor frame which diagonalises $g(r)$, has $X_1$ lying in the direction of $u_1$ and has $X_2,X_3$ spanned by $u_2,u_3$; we aim to show that, for Ricci solitons, this basis can be chosen independently of $r\in I$. For a fixed $r_0\in I$, choose a diagonalising Milnor frame for $g(r_0)$ so that $X_1$ lies in the direction of $u_1$, and $X_2,X_3$ is spanned by $u_2,u_3$. Proposition \ref{homogeneousdiagonal} implies that $g(r)(u_1,u_3)=g(r)(u_1,u_2)=0$ for all other values of $r\in I$, so in the Milnor frame $X_1,X_2,X_3$, we have that $g'_{12}(r_0)=g'_{13}(r_0)=0$. By rotating the $X_2,X_3$ basis, we can also assume that $g'_{23}(r_0)=0$ as well, thanks to the equation $(g^{22}(r_0)-g^{33}(r_0))g'_{23}(r_0)=0$. Thus, we can again use the ODEs \eqref{solitonODESgeneric} to conclude that this new basis $X_1,X_2,X_3$ globally diagonalises $g(r)$ for Ricci solitons. 

If $n\in \mathbb{N}\setminus\{1,2,4\}$ and we have a gradient Ricci soliton of the form \eqref{staticmetric} so that $X_i=\frac{u_i}{2}$ and $f_2=f_3$ uniformly, then Proposition \ref{homogeneousdiagonal} implies that the metric passes to the quotient. On the other hand, if the metric passes to the quotient, then for any $r\in I$, the metric $g(r)$ is diagonal in the Milnor frame $\{\frac{u_1}{2},\frac{u_2}{2},\frac{u_3}{2}\}$, and the metric is a scalar multiple of the identity in the basis $\{\frac{u_2}{2},\frac{u_3}{2}\}$, so $f_2=f_3$ uniformly.
\end{proof}
Note that a large part of this proof is essentially identical to an analogous proof for Einstein metrics presented in Section 3 of \cite{Dammerman}, although the author suggests that this computation has been well-known in the physics community for quite some time. 
\subsection{Smoothness at the singular orbit $SU(2)/U(1)=\mathbb{S}^2$}
The point $r=0$ corresponds to the two-dimensional singular orbit $SU(2)/U(1)$. We will now recall some of the results of \cite{VZ}, which tell us when our Riemannian metric \eqref{staticmetric} on $(0,\infty)\times SU(2)/\mathbb{Z}_n$ can be smoothly extended to include this orbit. First, the functions $f_1^2,f_2^2,f_3^2:(0,\infty)\to (0,\infty)$ must all be continuously extendible to functions on $\mathbb{R}$. Second, two of the functions must be positive at $r=0$, while the third vanishes, since the singular orbit $SU(2)/U(1)$ is two-dimensional, and furthermore, the corresponding basis vector $X_i$ must generate the $U(1)$ subgroup. Thus, after taking one more change of diagonalising Milnor frame, we can assume that for metrics of the form \eqref{staticmetric} on $SU(2)\times_{U(1)}\mathbb{R}^2$, the $X_1$ vector generates $U(1)$ (we already knew this for $n\ge 3$, but now we can assume it is true for $n=1,2$ as well).
Then \cite{VZ} gives us the following smoothness conditions:
\begin{thm}
The Riemannian metric \eqref{staticmetric} defined on $(0,\infty)\times SU(2)/\mathbb{Z}_n$ can be interpreted as a smooth Riemannian metric on the principal part of $M=SU(2)\times_{U(1)} \mathbb{R}^2$. 
The metric is smoothly extendable to the singular $\mathbb{S}^2$ orbit if and only if there are smooth functions $\phi_1,\phi_2,\phi_3:\mathbb{R}\to \mathbb{R}$ with $\phi_2(0)>0$ so that 
\begin{align}\label{VZSmoothness}
f_1(r)^2=\frac{n^2r^2}{4}+r^4 \phi_1(r^2), \qquad f_2(r)^2=\phi_2(r^2)+r^{\frac{4}{n}}\phi_3(r^2), \qquad f_3(r)^2=\phi_2(r^2)-r^{\frac{4}{n}}\phi_3(r^2).
\end{align}
\end{thm}
\begin{rmk}
If $n\in \mathbb{N}\setminus\{1,2,4\}$, then we know that $\phi_3$ must vanish uniformly, by Proposition \ref{homogeneousdiagonal}. 
\end{rmk}
We now conclude this preliminary section by summarising the results we have seen:
\begin{thm}
For any $SU(2)$-invariant gradient Ricci soliton on $M=SU(2)\times_{U(1)}\mathbb{R}^2$ (with the $U(1)$ action on $\mathbb{R}^2$ determined by $U(1)/\mathbb{Z}_n=\mathbb{S}^1$), there are three left-invariant vector fields $X_1,X_2,X_3$ on $SU(2)$ satisfying \eqref{Milnorframe} so that the metric has the form \eqref{staticmetric}, the $X_1$ vector field generates $U(1)$, the corresponding $f_i$ functions satisfy \eqref{VZSmoothness} and $f_2=f_3$ if $n\in \mathbb{N}\setminus\{1,2,4\}$, and the metric is a gradient Ricci soliton on the principal part $(0,\infty)\times SU(2)/\mathbb{Z}_n$. Conversely, any metric satisfying all of these conditions will be a smooth gradient Ricci soliton on $M$. 
\end{thm}
\section{The initial value problem}
A computation (c.f. \cite{GroveZiller}) reveals that the metric \eqref{staticmetric} defined on $I\times SU(2)/\mathbb{Z}_n$ satisfies the gradient Ricci soliton equation with Einstein constant $\lambda$ and $SU(2)$-invariant potential function $u$ if and only if the following ordinary differential equations hold:
\begin{align}\label{solitonequationsfirst}
\begin{split}
 -\frac{f_1''}{f_1}-\frac{f_2''}{f_2}-\frac{f_3''}{f_3}+u''&=\lambda,\\
 -\frac{f_1''}{f_1}+\frac{f_1'}{f_1}\left(u'-\frac{f_2'}{f_2}-\frac{f_3'}{f_3}\right)+\frac{f_1^4-(f_2^2-f_3^2)^2}{2(f_1f_2f_3)^2}&=\lambda,\\
-\frac{f_2''}{f_2}+\frac{f_2'}{f_2}\left(u'-\frac{f_1'}{f_1}-\frac{f_3'}{f_3}\right)+\frac{f_2^4-(f_1^2-f_3^2)^2}{2(f_1f_2f_3)^2}&=\lambda,\\
-\frac{f_3''}{f_3}+\frac{f_3'}{f_3}\left(u'-\frac{f_2'}{f_2}-\frac{f_1'}{f_1}\right)+\frac{f_3^4-(f_2^2-f_1^2)^2}{2(f_1f_2f_3)^2}&=\lambda.
\end{split}
\end{align}
In order to understand solutions of this equation, we also find it convenient to consider other forms of the equations that arise by changing variables. For our first change, we take $\xi=\frac{f_1'}{f_1}+\frac{f_2'}{f_2}+\frac{f_3'}{f_3}-u'$, $L_i=\frac{f_i'}{f_i}$, $R_i=\frac{f_i}{f_jf_k}$ (with $i,j,k\in \{1,2,3\}$ chosen to be pairwise distinct) so that \eqref{solitonequationsfirst} becomes 
\begin{align}\label{xifirstchange}
\begin{split}
\xi'&=-L_1^2-L_2^2-L_3^2-\lambda,\\
L_i'&=-\xi L_i+\frac{R_i^2}{2}-\frac{(R_j-R_k)^2}{2}-\lambda,\\
R_i'&=R_i(L_i-L_j-L_k).
\end{split}
\end{align}
This change is convenient for constructing short-time solutions of \eqref{solitonequationsfirst} which satisfy the smoothness conditions \eqref{VZSmoothness}. 
\begin{thm}\label{shorttimen4}
If $n=4$, there is a continuous function $\phi:\mathbb{S}^2\to (1,\infty)\times \mathbb{R}^6$ with the following properties:
\begin{itemize}
\item there is a $T>0$ so that for each $(\alpha,\beta,\gamma)\in \mathbb{S}^2$, there is a solution of \eqref{xifirstchange} on $(0,T]$, with data at $r=T$ coinciding with $\phi(\alpha,\beta,\gamma)$;
\item the solution satisfies \begin{align}\label{abgn4}
\alpha=\lim_{r\to 0}\frac{\xi(r)-L_1(r)-L_2(r)-L_3(r)}{r}, \qquad \beta=\lim_{r\to 0}\frac{R_1(r)}{r}, \qquad \gamma=\lim_{r\to 0}\frac{L_2(r)-L_3(r)}{2};
\end{align}
\item if $\alpha=0$, then $\xi=L_1+L_2+L_3$ and $\sum_{i=1}^{3}\frac{R_i^2}{2}-\frac{(R_j-R_k)^2}{2}+L_i^2=2\lambda+\xi^2$  uniformly (this is the Einstein case);
\item if $\beta=0$, then $R_1=0$ uniformly and if $\beta>0$, the corresponding solution $(u',f_1,f_2,f_3)$ of \eqref{solitonequationsfirst} is such that $u'$ is a smooth and odd function around $r=0$, and the metric components $f_1,f_2,f_3$  satisfy the smoothness conditions of \eqref{VZSmoothness};
\item if $\gamma=0$, then $L_2=L_3$ and $R_2=R_3$ uniformly.
\end{itemize}
\end{thm}
\begin{proof}
Choose $(\alpha,\beta,\gamma)\in \mathbb{S}^2$. We look for solutions of the form
\begin{align}\label{n4etaic}
\begin{split}
\xi= \frac{1}{r}+\eta_0, \qquad L_1= \frac{1}{r}+\eta_1, \qquad L_2= \gamma+\eta_2, \qquad L_3= -\gamma+\eta_3,\\ R_1= \eta_4, \qquad R_2= \frac{1}{2r}+\gamma+\eta_5, \qquad R_3= \frac{1}{2r}-\gamma+\eta_6,
\end{split}
\end{align}
where $\eta_i$ are smooth functions satisfying $\eta_i(0)=0$. 
The equation for $\eta=(\eta_0,\cdots,\eta_6)$ becomes 
\begin{align}\label{etan4}
\eta'(r)&=\frac{A \eta(r)}{r}+K+B(\gamma,\eta(r)), 
\end{align}
where 
\begin{align*}
A=\begin{pmatrix}
0&-2&0&0&0&0&0\\
-1&-1&0&0&0&0&0\\
0&0&-1&0&\frac{1}{2}&\frac{1}{2}&-\frac{1}{2}\\
0&0&0&-1&\frac{1}{2}&-\frac{1}{2}&\frac{1}{2}\\
0&0&0&0&1&0&0\\
0&-\frac{1}{2}&\frac{1}{2}&-\frac{1}{2}&0&-1&0\\
0&-\frac{1}{2}&-\frac{1}{2}&\frac{1}{2}&0&0&-1
\end{pmatrix}, \qquad K=\begin{pmatrix}
-\lambda-2\gamma^2\\
-\lambda-2\gamma^2\\
-\lambda\\
-\lambda\\
0\\
2\gamma^2\\
2\gamma^2
\end{pmatrix},
\end{align*}
and $B$ is smooth with $B(\gamma,0)=0$. 
Note that $A=PDP^{-1}$, where 
\begin{align*}
D&=\begin{pmatrix}
-2&0&0&0&0&0&0\\
0&-2&0&0&0&0&0\\
0&0&-1&0&0&0&0\\
0&0&0&-1&0&0&0\\
0&0&0&0&0&0&0\\
0&0&0&0&0&1&0\\
0&0&0&0&0&0&1\\
\end{pmatrix},
\end{align*}
and 
\begin{align*}
P=\begin{pmatrix}
1&1&0&0&0&8&0\\
1&1&0&0&0&-4&0\\
\frac{1}{2}&-\frac{1}{2}&0&1&-1&0&\frac{1}{4}\\
-\frac{1}{2}&\frac{1}{2}&0&1&1&0&\frac{1}{4}\\
0&0&0&0&0&0&1\\
0&1&1&0&-1&1&0\\
1&0&1&0&1&1&0
\end{pmatrix},
\qquad & P^{-1}=\begin{pmatrix}
\frac{1}{6}&\frac{1}{3}&\frac{1}{4}&-\frac{1}{4}&0&-\frac{1}{4}&\frac{1}{4}\\
\frac{1}{6}&\frac{1}{3}&-\frac{1}{4}&\frac{1}{4}&0&\frac{1}{4}&-\frac{1}{4}\\
-\frac{1}{4}&-\frac{1}{4}&0&0&0&\frac{1}{2}&\frac{1}{2}\\
0&0&\frac{1}{2}&\frac{1}{2}&-\frac{1}{4}&0&0\\
0&0&-\frac{1}{4}&\frac{1}{4}&0&-\frac{1}{4}&\frac{1}{4}\\
\frac{1}{12}&-\frac{1}{12}&0&0&0&0&0\\
0&0&0&0&1&0&0
\end{pmatrix}.
\end{align*}
Using $\tilde{\eta}=P^{-1}\eta$, we then get 
\begin{align*}
\tilde{\eta}'=\frac{D\tilde{\eta}}{r}+P^{-1}K+P^{-1}(\gamma,P\tilde{\eta}). 
\end{align*}
We observe that the last two entires of $P^{-1}K$ both vanish. 
Therefore, a standard existence and uniqueness argument for singular ODEs can be used to show the existence of a short-time smooth solution $\tilde{\eta}$ with $\tilde{\eta}(0)=0$, $\tilde{\eta}'_i(0)$ determined by a function of $\lambda,\gamma$ for $i=0,\cdots,5$, and $\tilde{\eta}'_6(0),\tilde{\eta}'_7(0)$ arbitrary. We can choose $\tilde{\eta}'_7(0)=\beta$ and subsequently vary $\tilde{\eta}'_6(0)$ uniquely to satisfy \eqref{abgn4}. We can find a uniform time $T$ so that all solutions with  $(\alpha,\beta,\gamma)\in \mathbb{S}^2$ extend to this time, and the data $(\xi(T),L_i(T),R_i(T))$ depends continuously on $(\alpha,\beta,\gamma)\in \mathbb{S}^2$. We can also insist that $\xi(T)>1$ since $\xi$ starts at $+\infty$.

Since the $\eta_i$ functions are all smooth, the functions $\xi,L_i,R_i$ all have Laurent series expansions, with lowest order terms being $\frac{1}{r}$. Putting these expressions into \eqref{xifirstchange} allows us to prove that, for each $k\in \mathbb{N}$, the following holds, by induction on $k$:
\begin{itemize}
\item if $k$ is odd, then $\eta_2^{(k)}-\eta_3^{(k)}=\eta^{(k)}_5-\eta_6^{(k)}=0$ at $r=0$;
\item if $k$ is even, then $\eta_0^{(k)}=\eta_1^{(k)}=\eta_2^{(k)}+\eta_3^{(k)}=\eta_4^{(k)}=\eta^{(k)}_5+\eta_6^{(k)}=0$ at $r=0$.
\end{itemize}
Thus, the Laurent series expansions of $\frac{1}{R_2}-\frac{1}{R_3}$, $\frac{1}{R_2}+\frac{1}{R_3}$ and $\frac{1}{R_1}$ consist only of even, odd and odd powers respectively (the expression $\frac{1}{R_1}$ is defined away from $r=0$ if $\beta\neq 0$). Thus, the equalities 
\begin{align*}
f_3^2-f_2^2=\frac{1}{R_1}\left(\frac{1}{R_2}-\frac{1}{R_3}\right), \qquad f_2^2+f_3^2=\frac{1}{R_1}\left(\frac{1}{R_2}+\frac{1}{R_3}\right)
\end{align*}
give the existence of the required smooth functions $\phi_2,\phi_3$ in \eqref{VZSmoothness}. Also, the fact that $\eta_1$ is a smooth odd function with $f_1'=f_1\left(\frac{1}{r}+\eta_1\right)$ and $f_1(0)=0$ implies that $f_1$ is a smooth and odd function. We also have $f_1'(0)=2$ because of $\lim_{r\to 0}rR_2(r)=\lim_{r\to 0}rR_3(r)=\frac{1}{2}$; the existence of the smooth function $\phi_1$ follows. Finally, the fact that $\eta_0,\eta_1,\eta_2+\eta_3$ are all smooth and odd implies that $u'=L_1+L_2+L_3-\xi$ is smooth and odd, as required. 

Finally, we characterise the three parameters $\alpha,\beta,\gamma$. If $\beta=0$, then $R_1'(0)=0$, and the equality $R_1'=R_1(\frac{1}{r}+\eta_1-\eta_2-\eta_3)$ implies that $R_1=0$ uniformly. If $\beta>0$, then $R_1(r)>0$ for $r>0$ by the same equation. If $\gamma=0$, then $L_2-L_3$ and $R_2-R_3$ are both zero initially, and the equations 
\begin{align*}
(L_2-L_3)'=-\xi (L_2-L_3)+(R_2-R_3)(R_2+R_3-R_1), \qquad (R_2-R_3)'=-L_1(R_2-R_3)+(R_2+R_3)(L_2-L_3)
\end{align*}
can  be written as 
\begin{align*}
\begin{pmatrix}
L_2-L_3\\
R_2-R_3
\end{pmatrix}'=\left(\frac{1}{r}
\begin{pmatrix}
-1&1\\
1&-1
\end{pmatrix}+B(r)\right)\begin{pmatrix}
L_2-L_3\\
R_2-R_3
\end{pmatrix},
\end{align*}
where $B(r)$ is smooth. Since the square matrix $\begin{pmatrix}
-1&1\\
1&-1
\end{pmatrix}$ has eigenvalues $0$ and $-2$, we conclude that $L_2=L_3$ and $R_2=R_3$ uniformly. 

We finally turn to the characterisation of $\alpha$. For each $\beta,\gamma$ we can use another ODE existence and uniqueness argument to find a solution of 
\begin{align}\label{xifirstchangeRicciflat}
\begin{split}
L_i'&=-(L_1+L_2+L_3) L_i+\frac{R_i^2}{2}-\frac{(R_j-R_k)^2}{2}-\lambda,\\
R_i'&=R_i(L_i-L_j-L_k),
\end{split}
\end{align}
which satisfies $\beta=\lim_{r\to 0}\frac{R_1(r)}{r}$ and $\gamma=\lim_{r\to 0}\frac{L_2(r)-L_3(r)}{2}$, and with $L_i,R_i$ satisfying the conditions in \eqref{n4etaic}. For this solution, we find that the quantity $Z=\left(\sum_{i=1}^{3}\frac{R_i^2}{2}-\frac{(R_j-R_k)^2}{2}+L_i^2\right)-2\lambda-(L_1+L_2+L_3)^2$ satisfies 
\begin{align}\label{Bianchi}
\frac{Z'}{2}=-(L_1+L_2+L_3)Z.
\end{align} Now, by expanding the expression for $Z$ using \eqref{n4etaic}, we find that all of the $\frac{1}{r},\frac{1}{r^2}$ terms cancel, so $Z$ is bounded as $r\to 0$, so by \eqref{Bianchi} and the fast blow-up of $L_1$, $Z$ must in fact vanish uniformly. By defining $\xi=L_1+L_2+L_3$, we thus obtain a solution of \eqref{xifirstchange}, \eqref{n4etaic} and \eqref{abgn4}, with $\alpha=0$; by uniqueness, this solution must co-incide with the solution we have already obtained. 
\end{proof}
A similar result holds for $n=1,2$; the only significant difference is that $\gamma$ characterises a higher order decay of $L_2-L_3$ and $R_2-R_3$. 
For other values of $n\in \mathbb{N}$, the following result concerning $U(2)$-invariant solutions is standard:
\begin{thm}\label{shorttimen}
For each $n\in \mathbb{N}$, there exists a continuous function $\phi:\mathbb{S}^1\to (1,\infty)\times \mathbb{R}^6$ with the following properties:
\begin{itemize}
\item there is a $T>0$ so that for each $(\alpha,\beta)\in \mathbb{S}^1$, there is a solution of \eqref{xifirstchange} with $L_2=L_3$ and $R_2=R_3$ on $(0,T]$, with data at $r=T$ coinciding with $\phi(\alpha,\beta)$;
\item the solution satisfies \begin{align}\label{abgn}
\alpha=\lim_{r\to 0}\frac{\xi(r)-L_1(r)-2L_2(r)}{r}, \qquad \beta=\lim_{r\to 0}\frac{R_1(r)}{r};
\end{align}
\item if $\alpha=0$, then $\xi=L_1+L_2+L_3$ and $L_1^2+2L_2^2+2R_1R_2-\frac{R_1^2}{2}-(L_1+2L_2)^2=2\lambda$ uniformly;
\item if $\beta=0$, then $R_1=0$ uniformly, and if $\beta>0$, then the corresponding solution of \eqref{solitonequationsfirst} satisfies the smoothness conditions of \eqref{VZSmoothness} at $r=0$. 
\end{itemize}
If $n=4$, the function $\phi(\alpha,\beta)$ coincides with $\phi(\alpha,\beta,0)$ from Theorem \ref{shorttimen4}. 
\end{thm}
With our short-time existence results in hand, we now introduce our second change of variables: $\frac{1}{\mathcal{L}}=\frac{f_1'}{f_1}+\frac{f_2'}{f_2}+\frac{f_3'}{f_3}-u'$, $X_i=\frac{\mathcal{L}f_i'}{f_i}$ and $Y_i=\frac{\mathcal{L}f_i}{f_jf_k}$, with $i,j,k$ again chosen to be pairwise distinct. Then 
\begin{align}\label{XYB4Chain}
\begin{split}
\mathcal{L}\mathcal{L}'&=\mathcal{L}\left(X_1^2+X_2^2+X_3^2+\lambda \mathcal{L}^2\right),\\
\mathcal{L}X_i'&=\frac{Y_i^2}{2}-\frac{(Y_j-Y_k)^2}{2}-X_i-\lambda\mathcal{L}^2+X_i(X_1^2+X_2^2+X_3^2+\lambda \mathcal{L}^2),\\
\mathcal{L}Y_i'&=Y_i(X_i-X_j-X_k+X_1^2+X_2^2+X_3^2+\lambda \mathcal{L}^2),
\end{split}
\end{align}
with $j,k$ as before. The solutions we have constructed in Theorems \ref{shorttimen4} and \ref{shorttimen} satisfy the conditions of \eqref{VZSmoothness}, so in these new variables, we have
\begin{align}\label{smoothnessXY}
\lim_{r\to 0}\frac{\mathcal{L}(r)}{r}=1, \qquad \lim_{r\to 0}X_1(r)=1, \qquad \lim_{r\to 0}X_2(r)=\lim_{r\to 0}X_3(r)=0, \qquad \lim_{r\to 0}Y_1(r)=0, \qquad \lim_{r\to 0}Y_2(r)=\lim_{r\to 0}Y_3(r)=\frac{2}{n}. 
\end{align}
The condition $\lim_{r\to 0}\frac{\mathcal{L}(r)}{r}$ implies that we can define a function $r(s)$ so that $r'(s)=\mathcal{L}(r(s))$,  $r(0)=T$ and $\lim_{s\to -\infty}r(s)=0$, so we arrive at the system
\begin{align}\label{XYEquations}
\begin{split}
\frac{d\mathcal{L}}{ds}&=\mathcal{L}\left(X_1^2+X_2^2+X_3^2+\lambda \mathcal{L}^2\right),\\
\frac{dX_i}{ds}&=\frac{Y_i^2}{2}-\frac{(Y_j-Y_k)^2}{2}-X_i-\lambda\mathcal{L}^2+X_i(X_1^2+X_2^2+X_3^2+\lambda \mathcal{L}^2), \ i=1,2,3,\\
\frac{d Y_i}{ds}&=Y_i(X_i-X_j-X_k+X_1^2+X_2^2+X_3^2+\lambda \mathcal{L}^2), \ i=1,2,3;
\end{split}
\end{align}
our solutions tend to the point $(0,1,0,0,0,\frac{2}{n},\frac{2}{n})$ as $s$ tends to $-\infty$. This system is particularly useful if $\lambda=0$, because then the last six equations do not feature $\mathcal{L}$ explicitly. 
We will use this system of equations, as well as \eqref{xifirstchange}, to study the long-time behaviour of our solutions. In particular, these two systems will help us to determine when the corresponding gradient solitons are complete. 

We conclude this section by detailing the extent to which solutions of our initial value problem can be extended.
\begin{prop}\label{Extend}
If $\lambda\ge 0$, a solution of \eqref{xifirstchange} on $[T,S)$ with $S<\infty$ which cannot be extended past $S$ has $\lim_{r\to S}\xi(r)=-\infty$. 
\end{prop}
\begin{proof}
Note that $\xi$ is monotone decreasing, so if $\lim_{r\to S}\xi(r)\neq -\infty$, then $\xi$ is uniformly bounded on $[T,S)$. Thus, if $S$ is a singular time, then $R_i$ or $L_i$ must become unbounded around $S$ for some $i=1,2,3$. The equation for $L_i'$ and the uniform boundedness of $\xi$ implies that if $L_i$ becomes unbounded for some $i$, then $R_i$ must be unbounded for some $i$. The equation for $R_i'$ implies that if $R_i$ becomes unbounded, then $\int_T^S\left|L_i\right|=\infty$ for some 
$i=1,2,3$. But then $\int_T^S L_i^2=\infty$ as well, which contradicts $\xi'=-L_1^2-L_2^2-L_3^2-\lambda$ and the uniform boundedness of $\xi$. 
\end{proof}
We thus find that a solution can always be extended as long as $\xi>0$; we will often use this fact without reference in the rest of the paper.


\section{$U(2)$-invariant solitons}
The solutions from Theorem \ref{shorttimen} have $R_2=R_3$ and $L_2=L_3$ everywhere, or equivalently, $Y_2=Y_3$ and $X_2=X_3$ everywhere. If $R_1,R_2,R_3>0$, the resulting Riemannian metric on $I\times SU(2)/\mathbb{Z}_n$ is also invariant under an action of $U(2)$ which acts transitively on the fibers $SU(2)/\mathbb{Z}_n$. 
In this section, we review and substantially simplify the construction of a number of complete $U(2)$-invariant Ricci solitons of the form \eqref{staticmetric} that arise from Theorem \ref{shorttimen}, as some of these are directly relevant to our construction of new $SU(2)$-invariant solitons. The equations in this case reduce to 
\begin{align}\label{xifirstchangebi}
\begin{split}
\xi'&=-L_1^2-2L_2^2-\lambda,\\
L_1'&=-\xi L_1+\frac{R_1^2}{2}-\lambda,\\
L_2'&=-\xi L_2+R_1R_2-\frac{R_1^2}{2}-\lambda,\\
R_1'&=R_1(L_1-2L_2),\\
R_2'&=-R_2L_1,
\end{split}
\end{align}
and 
\begin{align}\label{XYEquationsbi}
\begin{split}
\frac{d \mathcal{L}}{ds}&=\mathcal{L}(X_1^2+2X_2^2+\lambda \mathcal{L}^2),\\
\frac{dX_1}{ds}&=\frac{Y_1^2}{2}-X_1-\lambda\mathcal{L}^2+X_1(X_1^2+2X_2^2+\lambda \mathcal{L}^2),\\
\frac{dX_2}{ds}&=Y_1Y_2-\frac{Y_1^2}{2}-X_2-\lambda\mathcal{L}^2+X_2(X_1^2+2X_2^2+\lambda \mathcal{L}^2),\\
\frac{d Y_1}{ds}&=Y_1(X_1-2X_2+X_1^2+2X_2^2+\lambda \mathcal{L}^2),\\
\frac{d Y_2}{ds}&=Y_2(-X_1+X_1^2+2X_2^2+\lambda \mathcal{L}^2).
\end{split}
\end{align}
Theorem \ref{shorttimen} implies that there exists a family of solutions characterised by $(\alpha,\beta)\in \mathbb{S}^1$. In this section, we introduce the closed set $U=\{(\alpha,\beta)\vert \alpha^2+\beta^2=1, \alpha\ge 0, \beta\ge 0\}$ and study how varying $(\alpha,\beta)\in U$ can be used to produce complete solitons. We can ignore the case $\alpha<0$ because Proposition 2.3 of \cite{Dancer13} implies that $u'=L_1+L_2+L_3-\xi<0$ everywere on a complete steady soliton. 
\subsection{Existence of complete solitons}
In this subsection, we recover the existence of a family of $U(2)$-invariant steady ($\lambda=0$) solitons which were first constructed independently in \cite{Stolarski,Wink}. The first step is to study the case that $(\alpha,\beta)=(1,0)$. 
\begin{prop}\label{gammabeta0}
If $\lambda=0$, the solution of \eqref{XYEquationsbi} starting at $\phi(1,0)$ (described in Theorem \ref{shorttimen}) exists for all large $s$, the functions $\mathcal{L},X_1,X_2,Y_1,Y_2$ all have real-valued limits as $s$ tends to $\infty$, and moreover we have 
\begin{align*}
\lim_{s\to \infty}\mathcal{L}(s)>0, \qquad \lim_{s\to \infty}X_1(s)=\lim_{s\to \infty}X_2(s)=\lim_{s\to \infty}Y_1(s)=0, \qquad \lim_{s\to \infty}Y_2(s)>0. 
\end{align*}
\end{prop}
\begin{proof}
Since $\beta=0$, we know that $R_1$ vanishes uniformly by Theorem \ref{shorttimen}, so $Y_1$ also vanishes uniformly. In this case, the equations for $X_1,X_2$ become 
\begin{align*}
\frac{dX_1}{ds}&=X_1(X_1^2+2X_2^2-1),\\
\frac{dX_2}{ds}&=X_2(X_1^2+2X_2^2-1).
\end{align*}
Now the only solutions of this equation which approach $(1,0)$ as $s$ approaches $-\infty$ have $X_2=0$ uniformly. 
Since $\alpha>0$, we have that $X_1<1$ for some points close to the start of the solution, so we conclude that $X_1<1$ for the whole solution. Therefore, the solution converges exponentially quickly to $(0,0)$. This quick convergence of $X_1$ to $0$, the uniform vanshing of $X_2$ and the equations 
\begin{align*}
\frac{d Y_2}{ds}&=Y_2(-X_1+X_1^2+2X_2^2)\\
\frac{d \mathcal{L}}{ds}&=\mathcal{L}(X_1^2+2X_2^2)
\end{align*}
imply that the functions $\mathcal{L},Y_2$, which are positive everywhere, must converge to positive numbers. 
\end{proof}
Our next step is to perturb away from $(1,0)$ in $U$ to obtain more solutions of the soliton equations; since $\beta>0$, we can uniquely recover the positive functions $f_i$, thus obtaining a metric of the form \eqref{staticmetric} which satisfies the steady Ricci soliton equations and the smoothness conditions by Theorem \ref{shorttimen}. 
\begin{prop}\label{BiaxialOpen}
Fix $a>0$. For any $\theta_1\in (0,\frac{a}{2})$, there is a $\theta_2>0$ so that any trajectory of the last four equations of \eqref{XYEquationsbi} (with $\lambda=0$) starting $\theta_2$-close to $(0,0,0,a)$ stays $\theta_1$-close for all future times, and is convergent to $(0,0,0,b)$ for some $b>0$. Furthermore, the set $U_{BiaxialSolitons}\subseteq U$ of $(\alpha,\beta)$ so that the corresponding solution found in Theorem \ref{shorttimen} satisfies the conclusion of Proposition \ref{gammabeta0} is open in the relative topology induced by $U$.
\end{prop}
\begin{rmk}\label{gamma0betasmall}
 The set $U_{BiaxialSolitons}$ includes the point $(1,0)$ by Proposition \ref{gammabeta0}, even though this point does not, strictly speaking, give us a soliton metric on account of the uniform vanishing of $Y_1$.  
\end{rmk}
\begin{proof}
The point $(0,0,0,a)$ is a critical point of the last four equations of \eqref{XYEquationsbi} (with $\lambda=0$). This critical point is not hyperbolic. Indeed, the linearisation is 
\begin{align*}
\begin{pmatrix}
-1&0&0&0\\
0&-1&a&0\\
0&0&0&0\\
-a&0&0&0
\end{pmatrix},
\end{align*}
which has $-1$ and $0$ as its eigenvalues, each with two linearly-independent eigenvectors. The null eigenvectors are $(0,0,0,1)$ and $(0,a,1,0)$. Thus, there is a (not necessarily unique) two-dimensional invariant center manifold of class $C^3$ passing through $(0,0,0,a)$, whose tangent space is spanned by these null vectors. Since there are no eigenvalues with positive real part, the dynamics of this center manifold essentially determine \textit{all} of the local dynamics, in the sense that all trajectories close to $(0,0,0,a)$ will converge exponentially quickly to a center manifold trajectory. 

To study these central trajectories, we observe that the center manifold can (locally) be described by 
\begin{align*}
X_1=f_1(Y_1,Y_2-a), \qquad X_2=f_2(Y_1,Y_2-a),
\end{align*} where $f_1,f_2$  solve 
\begin{align*}
\frac{y_1^2}{2}-f_1+f_1(f_1^2+2f_2^2)&=\frac{\partial f_1}{\partial y_1}y_1\left(f_1-2f_2+f_1^2+2f_2^2\right)+\frac{\partial f_1}{\partial y_2}(y_2+a)(-f_1+f_1^2+2f_2^2),\\
y_1y_2+ay_1-\frac{y_1^2}{2}-f_2+f_2(f_1^2+2f_2^2)&=\frac{\partial f_2}{\partial y_1}y_1\left(f_1-2f_2+f_1^2+2f_2^2\right)+\frac{\partial f_2}{\partial y_2}(y_2+a)(-f_1+f_1^2+2f_2^2),
\end{align*}
and $f_1(y_1,y_2)=0$, $f_2(y_1,y_2)=ay_1$, correct to first order. We thus find that $\frac{f_1}{y_1}$ and $\frac{f_2}{y_1}$ can be continuously extended to functions in a neighbourhood of $(0,0,0,0)$. In fact, we then find that $\frac{f_1}{y_1^2}$ also admits a continuous extension, and tends to $\frac{1}{2}$ as $y_1,y_2$ tend to $0$. We also have that $\frac{f_2}{y_1}$ tends to $a$. These continuous extensions allow us to write equations for trajectories on the center manifold:
\begin{align*}
Y_1'=Y_1^2\left(-2a+O\left(\sqrt{Y_1^2+(Y_2-a)^2}\right)\right), \qquad Y_2'=Y_1^2Y_2O(1).
\end{align*}
This implies that, for trajectories starting on the center manifold, $\sup_{s\ge 0}\left(sY_1\right)$ tends to $0$ as the trajectory starts closer to $(0,0,0,a)$. For such a trajectory, $\sup_{s\ge 0}\left(s^2X_1^2+2s^2X_2^2\right)$ can also be made arbitrarily small by starting close to $(0,0,0,a)$ on the center manifold, so $Y_2$ is convergent to a number close to $a$. By the exponential convergence properties of the center manifold, we then find that \textit{any} solution starting close to $(0,0,0,a)$ has these same properties.

Now if $(\overline{\alpha},\overline{\beta})\in U_{BiaxialSolitons}$, then let $a=\lim_{s\to \infty}Y_2(s)$. For all nearby values of $(\alpha,\beta)$, the corresponding solution will end up close to the point $(0,0,0,a)$ \textit{somewhere} on its trajectory; by the previous argument, the solution will then converge to a point $(0,0,0,b)$ with $b$ close to $a$, and $\mathcal{L}$ will also converge to a positive number because of the quick decay of $X_1^2+X_2^2$. 
\end{proof}
The natural questions that arise from Proposition \ref{BiaxialOpen} are the following:
\begin{itemize}
\item How large can we make $\beta$ and still have a complete soliton?
\item If there is a `last' value of $\beta$, is there something special about the corresponding soliton?
\end{itemize}
As we will see, the answer to these questions depends on the value of $n\in \mathbb{N}$, and is largely determined by the behaviour of the solution with $(\alpha,\beta)=(0,1)$. To study these special solutions, recall from Theorem \ref{shorttimen} that  this forces $X_1+2X_2=1=X_1^2+2X_2^2+2Y_1Y_2-\frac{Y_1^2}{2}$ everywhere (this is the \textit{Ricci-flat} case). In the next subsections, we will examine these special trajectories and determine how they impact the existence of complete Ricci solitons. 
\subsection{$n=1$ and the Taub-Bolt metric}
 In case $n=1$, the resulting Ricci-flat metric is complete, and actually coincides with the well-known \textit{Taub-Bolt} metric. In our notation, this metric can be written in the (non-arclength parametrisation) form 
\begin{align*}
g_{Taub-Bolt}=\frac{r^2-1}{r^2-\frac{5r}{2}+1}dr\otimes dr+4 \left(\frac{r^2-\frac{5r}{2}+1}{r^2-1}\right)\omega_1\otimes \omega_1+\left(r^2-1\right)\left(\omega_2\otimes \omega_2+\omega_3\otimes \omega_3\right);
\end{align*}
this time the $\mathbb{S}^2$ singular orbit arises at $r=2$, and the metric is complete as $r\to \infty$.

\subsection{$n=2$ and the Eguchi-Hanson metric}
As with $n=1$, if we set $n=2$ and $(\alpha,\beta)=(1,0)$, we obtain a complete Ricci-flat metric which we identify as the K\"ahler Eguchi-Hanson metric:
\begin{align*}
g_{Eguchi-Hanson}=\frac{r^2}{\sqrt{1+r^4}}\left(dr\otimes dr+r^2\omega_1\otimes \omega_1\right)+\sqrt{1+r^4}(\omega_2\otimes \omega_2+\omega_3\otimes \omega_3).
\end{align*}

\subsection{$n\ge 3$ and Appleton's non-collapsed, steady solitons that are not Ricci-flat}
Interestingly, the Ricci-flat metrics that arise from setting $\alpha=0$ stop being complete when $n\ge 3$, as first demonstrated by Appleton \cite{Appleton}:
\begin{prop}\label{ng3complete}
If $n\ge 3$, then there is a neighbourhood of $(1,0)$ in $U$ in which all corresponding solutions of \eqref{xifirstchangebi} have 
$\xi<0$ for certain values of $r$. 
\end{prop}
\begin{rmk}
By Proposition 2.4 of \cite{Dancer13}, having $\xi<0$ implies that these metrics are not complete. 
\end{rmk}
\begin{proof}
The property that a solution achieves $\xi<0$ somewhere on its maximal interval of existence  survives perturbation, so it suffices to verify that the Ricci-flat metric corresponding to $(\alpha,\beta)=(0,1)$ eventually achieves $\xi<0$. 
This Ricci-flat metric corresponds to a trajectory of \eqref{XYEquationsbi} with $X_1+2X_2=1=X_1^2+2X_2^2+2Y_1Y_2-\frac{Y_1^2}{2}$. 
Therefore, it suffices to look at solutions of 
\begin{align}\label{Ricciflatbiaxial}
\begin{split}
\frac{dX_1}{ds}=\frac{Y_1^2}{2}+X_1(X_1^2+2(\frac{1-X_1}{2})^2-1),\\
\frac{d Y_1}{ds}=Y_1(2X_1-1+X_1^2+2(\frac{1-X_1}{2})^2).
\end{split}
\end{align}
By our smoothness conditions, our trajectory of interest starts at $(1,0)$. In fact, by using 
the equation $X_1^2+2X_2^2+2Y_1Y_2-\frac{Y_1^2}{2}=X_1^2+2(\frac{1-X_1}{2})^2+2Y_1Y_2-\frac{Y_1^2}{2}=1$ and the initial condition $\lim_{s\to -\infty}(X_1,X_2,Y_1,Y_2)=(1,0,0,\frac{2}{n})$, we find that 
\begin{align}\label{Ricciflatinitialasym}
\lim_{s\to -\infty}\frac{(1-X_1)}{Y_1}=\frac{2}{n}.
\end{align}
Now, the quantity 
$Z=X_1+Y_1-1$ is initially positive by \eqref{Ricciflatinitialasym}, and so the evolution equation $Z'=\frac{Z(3X_1^2+Z+1)}{2}$ implies that there is a finite $s^*$ around which $Z$ becomes arbitrarily large. Since $X_1,Y_1$ are both non-negative, we conclude that at least one of $X_1,Y_1$ becomes arbitrarily large as well. In fact, it must be the case that $X_1$ becomes large because, according to the second equation of \eqref{Ricciflatbiaxial}, $Y_1$ cannot become unbounded in finite $s$ without the same being true of $X_1$. Therefore, there comes a point in the soliton, after which $X_1>1$. The constraint $X_1+2X_2=1$ then implies that $X_2<0$ past this point as well. Going back to the equations of \eqref{xifirstchangebi}, we use the fact that $L_1>0$ and $L_2<0$ to conclude that $R_1>0$ is monotone increasing past this point on the soliton; the first two equations of \eqref{xifirstchangebi} then imply that $\xi$ eventually becomes negative. 
\end{proof}
Proposition \ref{ng3complete} implies that if $n\ge 3$, as we move away from the point $(1,0)$ in $U$, our `last' complete Ricci soliton is \textit{not} Ricci-flat. Appleton uses this crucial observation in \cite{Appleton} to construct non-collapsed Ricci solitons. Here, we prove a slight modification of that same result which will help us find new solitons that are \textit{not} $U(2)$-invariant.
\begin{thm}\label{AppletonSoliton}
If $n\ge 3$, then there exists a non-empty closed and connected set $U_0\subset U$ with the following properties:
\begin{itemize}
\item $U_0$ does not contain $(1,0)$ or $(0,1)$;
\item for any $(\alpha_0,\beta_0)\in U_0$, the corresponding solution of \eqref{XYEquationsbi} starting at $\phi(\alpha_0,\beta_0)$ exists for all large $s$;
\item the limits of $X_1,X_2,Y_1,Y_2,\mathcal{L}$ all exist and are equal to $0$, except for $\mathcal{L}$ which converges to a positive number;
\item  each open neighbourhood of $U_0$ in $U$ contains two points on either side, one of which is a complete Ricci soliton on the side connected to $(1,0)$ which satisfies the conclusion of Proposition \ref{gammabeta0}, and the other is an incomplete Ricci soliton on the side connected to $(0,1)$.
\end{itemize}
\end{thm}
\begin{rmk}
It is quite likely that $U_0$ consists only of a single point; this was conjectured in \cite{Appleton}. However, establishing this uniqueness result appears to be quite challenging.
\end{rmk}
\begin{proof}
The set $U=\{(\alpha,\beta)\in \mathbb{S}^1\vert \alpha,\beta\ge 0\}$ can be given a total ordering via the homemorphism sending $(\alpha,\beta)\in U$ to $\beta\in [0,1]$. Define $u_{start}\in U$ to be the supremum of all $(\alpha,\beta)\in U$ so that the corresponding solution exists for all $s\ge 0$, and satisfies the conclusion of Proposition \ref{gammabeta0}. Then by Proposition \ref{BiaxialOpen}, $u_{start}$ is the supremum of the non-empty set $U_{BiaxialSolitons}$, and $u_{start}>(1,0)$. A soliton corresponding to $(\alpha,\beta)\in U_{BiaxialSolitons}$ exists for all $r\ge 0$ and is complete because $\mathcal{L}$ converges to a positive number, so Proposition \ref{ng3complete} also implies that $u_{start}<(0,1)$. 

Now define $u_{final}\in U$ to be the infimum of all $(\alpha,\beta)\in [u_{start},(0,1)]$  so that the corresponding solution does not exist for all $r\ge 0$. The point $u_{final}$ exists and is under $(0,1)$ because of Proposition \ref{ng3complete}. It is clear that $u_{start}\le u_{final}$, so we define $U_0=[u_{start},u_{final}]$; by construction, any open neighbourhood of this closed set contains a solution satisfying the conclusion of Proposition \ref{gammabeta0}, and a solution which does not exist for all $r\ge 0$.

We claim that for any $(\alpha_0,\beta_0)\in U_0$, the solution of \eqref{xifirstchangebi} exists for all $r\ge 0$. Indeed, if $(\alpha_0,\beta_0)\in U_0\setminus\{ u_{final}\}$, then this follows from the infimum characterisation of $u_{final}$. On the other hand, if $(\alpha_0,\beta_0)=u_{final}$, then as we have just established, there is a sequence of points which converges to $(\alpha_0,\beta_0)$ from below, so the corresponding solutions exist for all large $r$ and must therefore have $\xi>0$, by Proposition 2.4 of \cite{Dancer13}. Therefore, the solution which occurs at $(\alpha_0,\beta_0)$ also have $\xi\ge 0$ on its maximal interval of existence, so Proposition \ref{Extend} implies that the solution exists for all $r\ge 0$. 

Now, for any $(\alpha_0,\beta_0)\in U_0$, the corresponding soliton is \textit{not} Ricci-flat, so by Proposition 2.4 of \cite{Dancer13}, we have 
\begin{align}\label{appletonlimitxi}
\lim_{r\to \infty}\xi(r)\in (0,\infty) \ \text{and} \ \lim_{s\to \infty}\mathcal{L}(s)\in (0,\infty).
\end{align}
We also know that this soliton has $L_1(r)>0$ for all $r\in (0,\infty)$ by the second equation of \eqref{xifirstchangebi}, so the fifth equation of \eqref{xifirstchangebi} implies that $R_2$, which is initially positive, is monotone decreasing, hence convergent, and 
\begin{align}\label{appletonlimitr2}
\lim_{r\to \infty}R_2(r)\in [0,\infty) \ \text{and} \ \lim_{s\to \infty}Y_2(s)\in [0,\infty).
\end{align}
Next we claim that $(X_1,X_2,Y_1,Y_2)$ is uniformly bounded. The second equation of \eqref{XYEquationsbi} implies immediately that $X_1\in [0,1]$. Now the fourth equation of \eqref{XYEquationsbi} gives $Y_1'\ge -Y_1$, so if $Y_1$ were unbounded, this slow decay would force $X_1$ past $1$, via the second equation. Thus, $Y_1$ is also uniformly bounded. Finally, the uniform bounds on $(X_1,Y_1,Y_2)$ that we already have combined with the third equation of \eqref{XYEquationsbi} imply that $X_2$ is also uniformly bounded. 

Our uniform bounds on $(\mathcal{L},X_1,X_2,Y_1,Y_2)$ imply uniform bounds on $(\xi,L_1,L_2,R_1,R_2)$, so the $\omega$-limit set for this trajectory is non-empty, and is compact. Statements \eqref{appletonlimitxi} and \eqref{appletonlimitr2} imply that 
this limit set lies in $$\{\lim_{r\to \infty}\xi(r)\}\times \mathbb{R}^3\times \{\lim_{r\to \infty}R_2(r)\}.$$ 
Invariance of the limit set and the first equation of \eqref{xifirstchangebi} then implies that the limit set must lie in 
 $$\{\lim_{r\to \infty}\xi(r)\}\times\{0\}\times \{0\}\times \mathbb{R}\times \{\lim_{r\to \infty}R_2(r)\},$$ 
which in turn implies (via the second equation of \eqref{xifirstchangebi}) that the $\omega$-limit set consists of the single point 
$$(\lim_{r\to \infty}\xi(r),0,0,0,\lim_{r\to \infty}R_2(r)).$$
Thus, the trajectory is convergent to this same point. 
We must have $\lim_{s\to \infty}Y_2(s)=0=\lim_{r\to \infty}R_2(r)$, because otherwise, we find that $(\alpha_0,\beta_0)$ is a member of the open set $U_{BiaxialSolitons}$, which contradicts the definition of $U_0$. 
\end{proof}
This concludes the construction of all known $U(2)$-invariant steady gradient Ricci solitons on $M=SU(2)\times_{U(1)}\mathbb{R}^2$. We finish this section with a proposition which gives us some uniform estimates on the asymptotics of these solitons which we find useful when constructing new solitons which are \textit{not} $U(2)$-invariant. 
\begin{prop}\label{uniformconvergenceu2}
For each $\delta'>0$, there is an open and connected neighbourhood $\tilde{U}\subset U$ of $U_0$ and a $T_0>0$ so that:
\begin{itemize}
\item if the solution $(\xi,L_1,L_2,R_1,R_2)$ corresponding to $(\alpha,\beta)\in\tilde{U}$ exists for all $r\ge 0$, then the solution $(X_1,X_2,Y_1,Y_2)$ lies in $B_{\frac{\delta'}{2}}(0)$ for all $s\ge T_0$; and 
\item for any $(\alpha,\beta)\in \tilde{U}$, the corresponding soliton exists up until time $T_0$, and $(X_1,X_2,Y_1,Y_2)$ lies in $B_{\frac{3\delta'}{4}}(0)$.
\end{itemize} 
\end{prop}
\begin{proof}
Fix $\delta'>0$, and make an initial choice of open connected neighbourhood $\tilde{U}\subset U$ containing $U_0$ so that the closure of $\tilde{U}$ does not contain $(0,1)$ or $(1,0)$. 
Recall from Section 2 of \cite{Dancer13} that the quantity 
\begin{align*}
C=2R_1R_2-\frac{R_1^2}{2}+L_1^2+2L_2^2-\xi^2
\end{align*}
is constant for solutions of \eqref{xifirstchangebi} with $\lambda=0$. This quantity clearly depends continuously on $(\alpha,\beta)\in U$, and is $0$ if and only if the soliton is Ricci-flat. Therefore, the exists a large $S_1>0$ so that $\frac{1}{C^2}+C^2\le S_1$ for all $(\alpha,\beta)\in \tilde{U}$. Also, by continuity of $\phi(\alpha,\beta)$, there is a $S_2>0$ so that $R_1+\frac{1}{R_1}+\xi +\frac{1}{\xi}\le S_2$ at time $r=T$ for all $(\alpha,\beta)\in \tilde{U}$. Therefore, by monotonicity of $\xi$ and the fact that $\lim_{r\to \infty}\xi(r)=\sqrt{-C}$ for complete solitons (Proposition 2.4 of \cite{Dancer13}), there exists a large $S_3$ so that 
\begin{align}
\frac{1}{\mathcal{L}(s)}+\mathcal{L}(s)+\frac{1}{Y_1(0)}\le S_3
\end{align}
for all $s\ge 0$ and all $(\alpha,\beta)\in \tilde{U}$ for which the corresponding soliton is complete. 

Next, we show that $Y_1\le Y_2$ for all $(\alpha,\beta)\in \tilde{U}$ for which the soliton is complete. Indeed, if this were not the case for a complete soliton, the equations 
\begin{align}
\left(\frac{Y_1}{Y_2}\right)'=2\frac{Y_1}{Y_2}(X_1-X_2), \qquad (X_1-X_2)'=(X_1-X_2)(X_1^2+2X_2^2-1)+Y_1^2\left(1-\frac{Y_2}{Y_1}\right)
\end{align}
imply that $\lim_{s\to 0}\frac{Y_1}{Y_2}$ exists, and lies in $(1,\infty]$. Using the same $\omega$-limit set techniques from the proof of Theorem \ref{AppletonSoliton}, we can conclude that any complete soliton has $\mathcal{L},X_1,X_2,Y_1,Y_2$ all convergent as $s$ tends to $\infty$, the $\mathcal{L}$ limit is positive, the $Y_2$ limit is non-negative, and all other limits are zero. In fact, $\lim_{s\to \infty}Y_2(s)=0$ because $\lim_{s\to \infty}\frac{Y_1(s)}{Y_2(s)}> 1$. Now, using 
\begin{align*}
\left(\frac{X_1}{Y_1^2}\right)'=\frac{1}{2}-\frac{X_1}{Y_1^2}\left(1+O(\sqrt{X_1^2+X_2^2})\right), \qquad \left(\frac{X_1+X_2}{Y_1Y_2}\right)'=1-\frac{X_1+X_2}{Y_1Y_2}\left(1+O(\sqrt{X_1^2+X_2^2})\right),
\end{align*}
we conclude that $\lim_{s\to \infty}\frac{X_1}{Y_1^2}=\frac{1}{2}$, and $\lim_{s\to \infty}\frac{X_1+X_2}{Y_1Y_2}=1$, so that $\lim_{s\to \infty}\left(3\frac{X_1}{Y_1^2}+\frac{Y_2}{Y_1}\left(\frac{X_1}{Y_1^2}-\frac{2(X_1+X_2)}{Y_1Y_2}\right)\right)=\frac{3}{2}-\lim_{s\to \infty}\frac{3Y_2}{2Y_1}>0$.
Then for large $s$, we have $Y_1>Y_2$ so that 
\begin{align*}
(Y_1-Y_2)'&=Y_1(X_1-2X_2)+X_1Y_2+(X_1^2+2X_2^2)(Y_1-Y_2)\\
&\ge Y_1^3\left(\frac{X_1}{Y_1^2}-\frac{2X_2}{Y_1^2}+\frac{X_1Y_2}{Y_1^3}\right)\\
&=Y_1^3\left(\frac{3X_1}{Y_1^2}-\frac{2(X_1+X_2)Y_2}{Y_2Y_1^2}+\frac{X_1Y_2}{Y_1^3}\right)\\
&=Y_1^3\left(3\frac{X_1}{Y_1^2}+\frac{Y_2}{Y_1}\left(\frac{X_1}{Y_1^2}-\frac{2(X_1+X_2)}{Y_1Y_2}\right)\right);
\end{align*}
we thus find that for large enough $s$, the quantity $Y_1-Y_2$ is increasing, positive and convergent to $0$, a contradiction. We thus conclude that $Y_1\le Y_2$ for complete solitons. 

Next, we show that there is a $\delta''>0$ with the following property: if $(\alpha,\beta)\in \tilde{U}$ gives a complete soliton, then once $(X_1,X_2,Y_1,Y_2)$ is $\delta''$-small, we have  $(X_1,X_2,Y_1,Y_2)\in B_{\frac{\delta'}{2}}(0)$ for all future times. Indeed, since $Y_1\le Y_2$ and $Y_2$ is monotone decreasing, we find that once $Y_1,Y_2$ are close to $0$, they stay that way. The exponential decay of $X_1,X_2$ then implies that if they are also initially close to zero, then they also stay that way. 

Now, for each $(\alpha_0,\beta_0)\in U_0$, there is an open neighbourhood lying in $\tilde{U}$ and a time $T_{(\alpha_0,\beta_0)}$ at which all solutions in this neighbourhood lie in $B_{\delta''}(0)$; by the previous step, the complete solitons in this neighbourhood then stay in $B_{\frac{\delta'}{2}}(0)$ thereafter. Since $U_0$ is compact, it can be covered by finitely many of these open neighbourhoods; we redefine $\tilde{U}$ to be the union of these neighbourhoods, and define $T_0$ to be the maximum of the values of $T_{(\alpha_0,\beta_0)}$ we get in this way. Then for any complete soliton arising from this new version of $\tilde{U}$, the data lies in $B_{\frac{\delta'}{2}}(0)$ past time $T_0$ because the data was in $B_{\delta''}(0)$ at some previous time. 
 Since solitons coming from $U_0$ are complete, the corresponding solutions lie in $B_{\frac{\delta'}{2}}(0)$ at time $T_0$, so by possibly shrinking $\tilde{U}$, we can assume that \textit{all} solitons from $\tilde{U}$ lie in $B_{\frac{3\delta'}{4}}(0)$ at time $T_0$ as well. 
\end{proof}
\section{$SU(2)$-invariant solitons for $n=4$}
In this section, we will construct new $SU(2)$-invariant steady gradient Ricci solitons that are \textit{not} $U(2)$-invariant. Our method essentially involves perturbing away from Appleton's solitons (constructed in Theorem \ref{AppletonSoliton}) in a non-$U(2)$-invariant direction.

To begin, recall from Theorem \ref{AppletonSoliton} that for $n\ge 3$, Appleton's soliton has $\lim_{s\to \infty}(X_1,X_2,X_3,Y_1,Y_2,Y_3)(s)=0$; the new $SU(2)$-invariant solitons are constructed by analysing the dynamics of \eqref{XYEquations} for values of $(X_1,X_2,X_3,Y_1,Y_2,Y_3)$ close to $0$. Note that the origin  
 is a non-hyperbolic critical point of the last six equations of \eqref{XYEquations}; the $6\times 6$ Jacobian derivative has $-1$ and $0$ as eigenvalues, both with three-dimensional eigenspaces. In fact, $\{e_i\}_{i=1}^{3}$ are three independent eigenvectors with eigenvalue $-1$, and $\{e_i\}_{i=4}^{6}$ are independent eigenvectors for the $0$ eigenvalue. 
Thus, the study of the dynamics is complicated by the existence of a three-dimensional \textit{center manifold}, which is described in the following theorem. 
\begin{thm}\label{descriptioncenter}
There is a $C^3$ function $C:\mathbb{R}^3\to \mathbb{R}$, and numbers $0<\delta '<\delta<\epsilon$ so that:
\begin{enumerate}
\item The three-dimensional submanifold $\mathcal{C}$ of $\mathbb{R}^6$ given by $x_1=C(y_1,y_2,y_3)$, $x_2=C(y_2,y_3,y_1)$ and $x_3=C(y_3,y_1,y_2)$ is invariant under the evolution of \eqref{XYEquations} (the last six equations with $\lambda=0$) for trajectories lying in $B_{\epsilon}(0)$ ($\mathcal{C}$ is referred to as the \textit{center manifold});
\item $C$ is symmetric in its final two entries, and correct to second order, we have $C(y_1,y_2,y_3)= \frac{y_1^2}{2}-\frac{(y_2-y_3)^2}{2}$;
\item The condition $y_i=y_j$ is preserved by \eqref{XYEquations} for trajectories that start in $\mathcal{C}$;
\item There is a continuous function $h:\mathbb{R}^6\to \mathcal{C}$ so that if $\phi(t,z)\in B_{\delta}(0)$ for all $t\ge 0$, then $\phi(t,h(z))\in B_{\epsilon}(0)$ for all $t\ge 0$, and $\phi(t,h(z))-\phi(t,z)$ converges to $0$ exponentially quickly; 
\item  For the same continuous function $h:\mathbb{R}^6\to \mathcal{C}$, if $\phi(t,h(z))\in B_{\delta}(0)$ for all $t\ge 0$ and $z\in B_{\delta}(0)$ then $\phi(t,z)\in B_{\epsilon}(0)$ for all $t\ge 0$, and $\phi(t,z)-\phi(t,h(z))$ converges to $0$ exponentially quickly (a converse of sorts to the previous point);
\item If $z=(x_1,x_2,x_3,y_1,y_2,y_3)\in B_{\delta'}(0)$ and $\max\{h_j(z)-h_i(z),h_k(z)-h_i(z)\}=0$ for any permutation $(i,j,k)$ of $(4,5,6)$, then $\phi(t,h(z))\in B_{\delta}(0)$ for all $t\ge 0$, and this center manifold solution satisfies:
\begin{itemize}
\item $\lim_{t\to \infty}\phi(t,h(z))=(0,0,0,a,b,c)$, where two of $a,b,c$ are identical and the third is zero (depending on $i,j,k$);
\item  $\sup_{t\ge 1}t^2\sum_{i=4}^{6}\phi_i(t,h(z))^2<\infty$.
\end{itemize} 
\end{enumerate}
\end{thm}
The proof of this theorem will be delayed until the next section. In the meantime, we will use it to construct our new solitons:
\begin{thm}
For $n=4$, there exists an open interval $I$ containing $0$ so that for each $\gamma\in I$, there is a pair $(\alpha,\beta)\in \mathbb{R}^2$ satisfying the following:
\begin{enumerate}
\item $\alpha^2+\beta^2+\gamma^2=1$;
\item the corresponding solution  $(\mathcal{L},X_1,X_2,X_3,Y_1,Y_2,Y_3)$ of \eqref{XYEquations} starting at $\phi(\alpha,\beta,\gamma)$ exists for all $s>0$ and is convergent;
\item we have $\lim_{s\to \infty}(\mathcal{L},X_1,X_2,X_3,Y_1,Y_2,Y_3)=(l,0,0,0,a,b,c)$, where $l>0$, and $a=b,c=0$ or $a=c,b=0$. 
\end{enumerate}
\end{thm}
\begin{proof}
Choose the neighbourhood $\tilde{U}$ of $U_0\subset U\subset \mathbb{S}^1$ and the $T_0$ given by Proposition \ref{uniformconvergenceu2} according to the $\delta'$ described in Theorem \ref{descriptioncenter}. Choose a closed and connected subset $\Omega$ of $\tilde{U}$ so that one of the boundary points of $\Omega$ corresponds to a complete soliton satisfying the conclusion of Proposition \ref{gammabeta0} on the side connected to $(0,1)$, and the other boundary corresponds to an incomplete soliton. We extend $\Omega$ to a closed and convex set in $\mathbb{S}^2$ with open interior so that the image of the continuous function $G:\Omega\to \mathbb{R}^6$ which records the data $(X_1,X_2,X_3,Y_1,Y_2,Y_3)$ at time $T_0$ has image lying in $B_{\delta'}(0)$ (such an extension is possible because of the second point of Proposition \ref{uniformconvergenceu2}). 
Define the continuous function $E:\mathbb{R}^6\to \mathbb{R}$ with $$E(x_1,x_2,x_3,y_1,y_2,y_3)=\max\{y_2-y_1,y_3-y_1\},$$ and the continuous function $F:\Omega\to \mathbb{R}$ with $F=E\circ h\circ G$, where $h$ is the center manifold function described in Theorem \ref{descriptioncenter}. 
We claim that if $F(\alpha,\beta,\gamma)=0$, then the corresponding solution of \eqref{XYEquations} satisfies points 2 and 3 of the statement of the theorem. Indeed, for such a choice of $(\alpha,\beta,\gamma)$, we find that the trajectory starting at $h(G(\alpha,\beta,\gamma))$ lies in $\mathcal{C}\cap B_{\delta}(0)$, for all $t\ge 0$, by point 6 of Theorem \ref{descriptioncenter}. For this solution, the condition $\sup_{t\ge 1}t^2\sum_{i=4}^{6}\phi_i(t,h(z))^2<\infty$ implies that $\mathcal{L}$, which is initially positive, must converge to a positive number $l$. Now by point 5 of Theorem \ref{descriptioncenter}, the trajectory starting from $G(\alpha,\beta,\gamma)$ exists and lies in $B_{\epsilon}(0)$ for all $t\ge 0$, and converges exponentially quickly to the convergent trajectory we have just discussed; this quick convergence implies that the key properties are translated to the new solution, as required. 

We have just established that it suffices to construct zeroes of $F(\alpha,\beta,\gamma)$ for all $\gamma$ in some open interval $I$ which contains $0$. To this end, we first show that the degree of the function sending points $(\alpha,\beta)\in \mathbb{S}^1$ to $F(\alpha,\beta,0)$ is non-zero when restricted to $\Omega\cap \{\gamma=0\}$ (the original definition of $\Omega$). Indeed, the complete soliton on the boundary of $\Omega$ lies in $B_{\frac{\delta'}{2}}(0)$ after time $T_0$ by Proposition \ref{uniformconvergenceu2}, so its center manifold counterpart lies in $B_{\epsilon}(0)$ for all $t\ge T_0$ (point 4 of Theorem \ref{descriptioncenter}), and must be exponentially converging to the original solution. Thus this center manifold solution must be converging to a point with $y_2=y_3>y_1=0$, so its starting point also had $y_2=y_3>y_1$ (point 3 of Theorem \ref{descriptioncenter}), so $F(\alpha,\beta,0)>0$ for this particular choice. 
On the other hand if $F(\alpha,\beta,0)\ge 0$ for our incomplete soliton on the boundary of $\Omega$ with $Y_2=Y_3$, then by point 6 of Theorem \ref{descriptioncenter}, the corresponding center manifold solutions lies in $B_{\delta}(0)$ for all $t\ge 0$ and gives rise to a complete soliton. Therefore, the original solution lies in $B_{\epsilon}(0)$ (point 5 of Theorem \ref{descriptioncenter}) for all $t\ge 0$ as well, and we get exponential convergence, and the soliton is also complete, which is a contradiction. Therefore, the Brouwer degree of $F$ restricted to $\Omega\cap \{\gamma=0\}$ is non-zero. 

By the homotopy invariance of the Brouwer degree, the fact that the degree is non-zero on $\Omega\cap \{\gamma=0\}$ implies that the same is true for $\Omega\cap \{\gamma=a\}$ for small values of $a$. Since the degree is non-zero, we obtain existence of zeroes for an open interval for $\gamma$, as required. 
\end{proof}

\section{Center manifolds}
This section is devoted to the rather technical proof of Theorem \ref{descriptioncenter}. The proof will follow from a sequence of lemmas which progressively construct the function $C$, and verify the points in the statement of the theorem. Most of the techniques come from the standard theory of center manifolds, (as discussed, for example, in \cite{Center}), but we ask for some additional technical requirements of our center manifold, so it seems appropriate to include some of the arguments. 
\begin{lem}
There is a $C^3$ function $C:\mathbb{R}^3\to \mathbb{R}$ satisfying points 1 and 2 of Theorem \ref{descriptioncenter}.
\end{lem}
\begin{proof}
We begin by  introducing the linear map $A:\mathbb{R}^6\to \mathbb{R}^6$ and a polynomial $g:\mathbb{R}^6\to \mathbb{R}^6$ with only second and third order terms so that the last six equations of \eqref{XYEquations} become
\begin{align}\label{cssystembefore}
z'=Az+g(z)
\end{align}
for $z$ a six-dimensional and time-varying vector. To make things explicit, note that 
\begin{align*}
Ae_i = \begin{cases}
-e_i \ &\text{if} \ i=1,2,3,\\
0 \ &\text{if} \ i=4,5,6,
\end{cases}
\end{align*}
and 
\begin{align*}
g(x_1,x_2,x_3,y_1,y_2,y_3)&=\begin{pmatrix}
\frac{y_1^2-(y_2-y_3)^2}{2}+x_1(x_1^2+x_2^2+x_3^2)\\
\frac{y_2^2-(y_1-y_3)^2}{2}+x_2(x_1^2+x_2^2+x_3^2)\\
\frac{y_3^2-(y_1-y_2)^2}{2}+x_3(x_1^2+x_2^2+x_3^2)\\
y_1(x_1-x_2-x_3+x_1^2+x_2^2+x_3^2)\\
y_2(x_2-x_3-x_1+x_1^2+x_2^2+x_3^2)\\
y_3(x_2-x_1-x_2+x_1^2+x_2^2+x_3^2)\\
\end{pmatrix}.
\end{align*}

Since we are only studying the local dynamics, we find it convenient to introduce a smooth and even bump function $\rho:\mathbb{R}\to [0,1]$ with 
\begin{align*}
\rho(r)=\begin{cases}
1 \ \text{if} \ \left|r\right|\le 1,\\
0 \ \text{if} \ \left|r\right|\ge 2,
\end{cases}
\end{align*}
and consider the equation
\begin{align}\label{cssystem}
z'=Az+\tilde{g}(z),
\end{align}
where $\tilde{g}(z)=\rho\left(\frac{\left|z\right|}{\epsilon}\right)g(z)$
for some small $\epsilon>0$. 
We can study the system \eqref{cssystem} as an alternative to \eqref{cssystembefore} because the two systems are the same for trajectories satisfying $\left|z\right|\le \epsilon$. 

Now our path to constructing the function $C$ is rather indirect, in the sense that we actually construct the center manifold $\mathcal{C}$ first. The theory regarding existence of center manifolds is quite standard, but we go through some of the details again. Denote with $\pi_s,\pi_c:\mathbb{R}^6\to \mathbb{R}^3$ the stable and center eigenspace projections and $i_s,i_c:\mathbb{R}^3\to \mathbb{R}^6$ the respective inclusion mappings. We define 
\begin{align*}
\mathcal{C}=\{z\in \mathbb{R}^6  \ \vert \  \sup_{t\in \mathbb{R}}e^{-\frac{\left|t\right|}{2}}\left|\phi(t,z)\right|<\infty\},
\end{align*}
where $\phi(t,z)$ means the flow of \eqref{cssystem} starting from $z$; this is well defined for all $t\in \mathbb{R}$ because of the global Lipschitz continuity of the system. 
This set $\mathcal{C}$ is clearly invariant under the flow of \eqref{cssystem}. We also claim that $\mathcal{C}$ is homeomorphic to $\mathbb{R}^3$, with inverse $\pi_c$. 
To see this, define the Banach space $$V=\{ v\in C^{0}(\mathbb{R}:\mathbb{R}^6) \ \vert\ \   \sup_{t\in \mathbb{R}}\left|v(t)\right|e^{-\frac{\left|t\right|}{2}}<\infty  \}$$ and consider the function $F:\mathbb{R}^3\times  V\to V$ with 
\begin{align*}
F(y,v)=i_c(y)+\int_0^t \pi_c \tilde{g}(v(s))ds+\int_{-\infty}^{t}e^{A(t-s)}\pi_s \tilde{g}(v(s))ds.
\end{align*}
The shrinking of $\epsilon$ forces the $C^1$ norm of $\tilde{g}$ to be small, implying that $F$ is a contraction. It then becomes clear that for each $y$, there is a unique fixed point $v$. Our homeomorphism is found by sending $y$ to the corresponding value of $v(0)$; this dependence is continuous, and $\pi_c(v(0))=y$ by construction. Also note that if $F(y,v)=v$, then $v'(t)=Av+\tilde{g}(v(t))$ so $v(0)\in \mathcal{C}$. 

Next we verify that the center manifold $\mathcal{C}$ can indeed be described using a single $C^3$ function $C:\mathbb{R}^3\to \mathbb{R}$. We already know there are three continuous functions $C_1,C_2,C_3:\mathbb{R}^3\to \mathbb{R}$ parametrising the center manifold via $x_i=C_i(y_1,y_2,y_3)$. It follows from the regular proof of the Center Manifold Theorem \cite{Center} that these functions are of class $C^3$. We now claim that
\begin{align*}
C_2(y_1,y_2,y_3)=C_1(y_2,y_1,y_3), \ C_3(y_1,y_2,y_3)=C_1(y_3,y_2,y_1), \  \text{and} \  C_3(y_1,y_2,y_3)=C_2(y_1,y_3,y_2) \ \text{for all} \  (y_1,y_2,y_3)\in \mathbb{R}^3.
\end{align*}
Indeed, we have  $$(C_1(y_1,y_2,y_3),C_2(y_1,y_2,y_3),C_3(y_1,y_2,y_3),y_1,y_2,y_3)\in \mathcal{C}$$
by construction, so that 
\begin{align*}
(C_2(y_1,y_2,y_3),C_1(y_1,y_2,y_3),C_3(y_1,y_2,y_3),y_2,y_1,y_3), \\ (C_3(y_1,y_2,y_3),C_2(y_1,y_2,y_3),C_1(y_1,y_2,y_3),y_3,y_2,y_1), \\ \text{and} \ (C_1(y_1,y_2,y_3),C_3(y_1,y_2,y_3),C_2(y_1,y_2,y_3),y_1,y_3,y_2)
\end{align*}
are all in $\mathcal{C}$
 as well by the symmetries of \eqref{cssystem}. 
The result then follows from the uniqueness of the center manifold that we have already established. The required $C^3$ function can then be defined with $C=C_1$.  The second order Taylor series for $C$ can be found from the observation that the three functions $C_1,C_2,C_3$ satisfy the system of partial differential equations 
\begin{align*}
x_i'(C_1,C_2,C_3,y_1,y_2,y_3)=\nabla C_i(y_1,y_2,y_3)\cdot y'(C_1,C_2,C_3,y_1,y_2,y_3)
\end{align*}
for $i=1,2,3$, coupled with the fact that these three functions are zero to first order. 
\end{proof}
We now have our candidate function that satisfies points 1 and 2. Point 3 follows from the observation that if a point $z=(x_1,x_2,x_3,y_1,y_2,y_3)\in \mathcal{C}$ has $y_i=y_j$, then $x_i=x_j$ by the symmetries of $C$; a point satisfying these two conditions retains these conditions as it is evolved by \eqref{cssystem} (this is essentially the $U(2)$-invariant case).

Next we consider points 4 and 5. 
\begin{lem}\label{dimensionreduction}
For the manifold $\mathcal{C}$ already constructed, there is a continuous function $h:\mathbb{R}^6\to \mathcal{C}$ satisfying conditions 4 and 5 of Theorem \ref{descriptioncenter} for some choice of $\delta>0$. 
\end{lem}
\begin{proof}
Define the Banach space 
\begin{align*}
W=\{w\in C^{0}(\mathbb{R}:\mathbb{R}^6) \ \vert \ \sup_{t\in \mathbb{R}}e^{\frac{t}{2}}\left|w(t)\right|< \infty\}.
\end{align*}
Note that $W\subset V$, where $V$ is the Banach space defined from the previous proof:
\begin{align*}
V=\{v\in C^{0}(\mathbb{R}:\mathbb{R}^6) \ \vert \ \sup_{t\in \mathbb{R}}e^{-\frac{\left|t\right|}{2}}\left|v(t)\right|< \infty\}.
\end{align*}
We will now define a continuous function $h:\mathbb{R}^6\to \mathcal{C}$ so that, for each $z\in \mathbb{R}^6$, the time-varying function $\phi(t,z)-\phi(t,h(z))$ converges to $0$ exponentially as $t$ increases. 
For each such $z$, we define the function $z^*:\mathbb{R}\to \mathbb{R}^6$ with
\begin{align*}
z^*(t)=\begin{cases}
\phi(t,z) \ &\text{for} \ t\ge 0, \\
z \ &\text{for} \ t\le 0;
\end{cases}
\end{align*}
clearly this function lies in $V$, since no eigenvalues of $A$ are positive, and $\tilde{g}$ is uniformly bounded. In fact, the norm of $z^*$ in $V$ becomes arbitrarily small by having $z$ small. 
Now consider the function $\Gamma: \mathbb{R}^6\times W\to W$ with 
\begin{align*}
\Gamma(z,w)=\int_{-\infty}^{t}e^{A(t-s)}\pi_s\left(\tilde{g}(z^*(s)+w(s))-\tilde{g}(z^*(s))-\varphi(s)\right)ds+\int_t^{\infty}\pi_c\left(\varphi(s)+\tilde{g}(z^*(s))-\tilde{g}(z^*(s)+w(s))\right)ds,
\end{align*}
where 
\begin{align*}
\varphi(t)=\begin{cases}
0 \ &\text{for} \ t\ge 0, \\
-Az-\tilde{g}(z) \ &\text{for} \ t< 0. 
\end{cases}
\end{align*}
Since $\epsilon>0$ is small, 
we find that $\Gamma$ is a contraction, so there is a unique fixed point $w\in W$. Then the function $v(t)=z^*(t)+w(t)$ solves equation \eqref{cssystem}, and is also in $
V$, so that $v(0)\in \mathcal{C}$ by definition of $\mathcal{C}$. The function $w$ clearly depends continuously (in the $W$ sense) on $z$, so the same is true for $v(0)$; we thus define $h(z)=v(0)$. 

By construction, the size of $w$ in $W$ can be made arbitrarily small by making $z\in \mathbb{R}^6$ small. We find that if $\phi(t,z)$ is small for all $t\ge 0$, then the same is true of $v(t)=\phi(t,h(z))$. On the other hand, if $z$ and $v(t)=\phi(t,h(z))$ are both small, then the same is true for $\phi(t,z)$; thus points 4 and 5 hold, after making an appropriate choice of $\delta>0$.
\end{proof}
Lemma \ref{dimensionreduction} implies that the dynamics of \eqref{XYEquations} close to the origin are essentially determined by the dynamics on the invariant and lower-dimensional center manifold $\mathcal{C}$. We conclude by shedding some light on some of these lower-dimensional dynamics, thus concluding the proof of Theorem \ref{descriptioncenter}. 
\begin{lem}\label{centerdichotomy}
There exists an appropriate choice of $\delta'\in (0,\delta)$ so that Property 6 holds for solutions starting in $B_{\delta'}(0)$.
\end{lem}
\begin{proof}
If $y_j=y_k$ for our center manifold solution, then properties 2 and 3 of Theorem \ref{descriptioncenter} (properties we have already established) imply that $x_j=x_k$ initially, and that these two equalities are preserved for the whole trajectory, i.e., we are looking at a $U(2)$-invariant soliton given by the equations 
\begin{align}\label{XYij}
\begin{split}
\frac{dX_i}{ds}&=\frac{Y_i^2}{2}-X_i+X_i(X_i^2+2X_j^2),\\
\frac{dX_j}{ds}&=Y_iY_j-\frac{Y_i^2}{2}-X_j+X_j(X_i^2+2X_j^2),\\
\frac{d Y_i}{ds}&=Y_i(X_i-2X_j+X_i^2+2X_j^2),\\
\frac{d Y_j}{ds}&=Y_j(-X_i+X_i^2+2X_j^2).
\end{split}
\end{align}By restricting our three-dimensional center manifold $\mathcal{C}\subset \mathbb{R}^6$ to the two-dimensional submanifold given by $y_j=y_k$, we find two $C^3$ functions $C_i,C_j$ of $y_i,y_j$, vanishing to first order, which satisfy the following PDEs:
\begin{align*}
C_i(C_i^2+2C_j^2-1)+\frac{y_i^2}{2}&=\frac{\partial C_i}{\partial y_i}y_i (C_i-2C_j+C_i^2+2C_j^2)+\frac{\partial C_i}{\partial y_j}y_j(-C_i+C_i^2+2C_j^2);\\
C_j(C_i^2+2C_j^2-1)+y_iy_j-\frac{y_i^2}{2}&=\frac{\partial C_j}{\partial y_i}y_i (C_i-2C_j+C_i^2+2C_j^2)+\frac{\partial C_j}{\partial y_j}y_j(-C_i+C_i^2+2C_j^2).
\end{align*}
These equations imply that $\frac{C_i}{y_i}$ and $\frac{C_j}{y_i}$ can be continuously extended to functions in a neighbourhood of $(0,0)$, which both vanish at this same point. In fact, we can then conclude that the function $\frac{C_i}{y_i^2}$ can be extended to a continuous function on a neighbourhood of $(0,0)$ satisfying $\frac{C_i}{y_i^2}=\frac{1}{2}+O(\sqrt{y_i^2+y_j^2})$, and that $\frac{C_j}{y_i}=y_j-\frac{y_i}{2}+O\left(y_i^2+y_j^2\right)$. Therefore, for small enough $y_i^2+y_j^2$, the quantity $-C_i+C_i^2+2C_j^2$ is non-negative, and is zero if and only if $y_i=0$.

We now have enough information to study the center manifold dynamics. Since $Y_j\ge Y_i$ initially, this is true for all times. Using \eqref{XYEquationsbi} and the $C_i,C_j$ functions, we can write
\begin{align*}
\frac{d Y_i}{ds}=Y_i(C_i-2C_j+C_i^2+2C_j^2), \qquad \frac{d Y_j}{ds}=Y_j(-C_i+C_i^2+2C_j^2)
\end{align*}
for trajectories on the center manifold which lie in $B_{\epsilon}(0)$. The fact that $Y_j\ge Y_i$, and $Y_j$ is decreasing implies that solutions which start close to $(0,0)$ will stay that way for all future times. Therefore, the invariant $\omega$-limit set exists and is connected. In fact, this limit set must consist of a single point $(0,a)$ for some small $a\ge 0$, so the trajectory converges to this point. It is clear that $a$ can be made as small as desired by having our solution start close to $(0,0)$. Now, if $a>0$, we compute 
\begin{align*}
\lim_{s\to \infty}\frac{Y_i'}{Y_i^2}=-2a,
\end{align*}
so that $Y_i$ converges to $0$ like $\frac{1}{s}$. By our estimates on $C_i,C_j$, we then conclude that $X_i^2+X_j^2$ converges to $0$ like $\frac{1}{s^2}$. If $a=0$, then we compute 
\begin{align*}
\frac{Y_i'}{Y_i^3}=\frac{3}{2}-\frac{Y_j}{Y_i}+K\le -\frac{1}{2}+K,
\end{align*}
where $K$ is small for $y_i^2+y_j^2$ small. Thus, $Y_i$ converges to $0$ at least as fast as $\frac{1}{\sqrt{s}}$. Then the observation 
\begin{align*}
\frac{Y_i'}{Y_i^2}=\frac{3 Y_i}{2}-Y_j+O(Y_i^2+Y_j^2), \qquad \frac{Y_j'}{Y_i^2}=-\frac{Y_j}{2}+O(Y_i^2+Y_j^2)
\end{align*}
and linear stablity analysis implies that $\frac{Y_i}{Y_j}$ converges to $1$, so $Y_j$ also vanishes like $\frac{1}{\sqrt{s}}$, so $X_i^2+X_j^2$ converges to $0$ like $\frac{1}{s^2}$, as required. 
\end{proof}


\begin{thebibliography}{99}


\bibitem{Appleton}
A. Appleton, \textit{A family of non-collapsed steady Ricci solitons in even dimensions greater or equal to four}, arXiv preprint arXiv:1708.00161.

\bibitem{Eguchi}
S. Brendle and N. Kapouleas, \textit{Gluing Eguchi‐Hanson metrics and a question of Page}, Comm. Pure Appl. Math. \textbf{70} (2017), no. 7, 1366--1401. 

\bibitem{Center}
A. Bressan, \textit{A tutorial on the center manifold theorem}, in \textit{Hyperbolic systems of balance laws 1911} (2003), 327--344.

\bibitem{Dancer13}
M. Buzano, A. Dancer and M. Wang, \textit{A family of steady Ricci solitons and Ricci-flat metrics}, Comm. Anal. Geom. \textbf{23} (2015), no. 3, 611--638. 


\bibitem{Dammerman}
B. Dammerman, \textit{Diagonalizing cohomogeneity-one Einstein metrics}, J. Geom. Phys. \textbf{59} (2009), 1271--1284. 


\bibitem{GroveZiller}
K. Grove and W. Ziller, \textit{Cohomogeneity one manifolds with positive Ricci curvature}, Invent. Math. \textbf{149} (2002), no. 3., 619--646. 

\bibitem{LocalBifurcations}
M. Haragus and G. Ioss, \textit{Local bifurcations, center manifolds, and normal forms in infinite-dimensional dynamical systems}, Springer, 2010. 


\bibitem{TaubBolt}
G. Holzegel, T. Schmelzer and C. Warnick, \textit{Ricci flows connecting Taub-Bolt and Taub-NUT metrics}, Classical Quantum Gravity \textbf{24} (2007), no. 24, 6201--6217. 

\bibitem{Stolarski}
M. Stolarski, \textit{Steady Ricci solitons on complex line bundles}, arXiv preprint, arXiv:1511.04087.  
\bibitem{VZ}
L. Verdiani and W. Ziller, \textit{Smoothness conditions in cohomogeneity one manifolds}, Transform. Groups (2020), doi: 10.1007/S00031-020-09618-9.

\bibitem{Wink}
M. Wink, \textit{Cohomogeneity one Ricci solitons from Hopf fibrations}, to appear in Comm. Anal. Geom., arXiv:1706.09712. 

\end{thebibliography}
\end{document}